\documentclass[11pt]{article}
\usepackage[a4paper,margin=25mm]{geometry}
\usepackage[T1]{fontenc}
\usepackage{amsmath,amssymb,amsthm}
\usepackage[colorlinks=true,linkcolor=blue,citecolor=blue,urlcolor=blue]{hyperref}
\newcommand{\R}{\mathbb R}
\newcommand{\C}{\mathbb C}
\newcommand{\D}{\mathcal D}
\newcommand{\K}{\mathsf K}
\newcommand{\curl}{\nabla\times}
\newcommand{\diver}{\nabla\cdot}
\newcommand{\Var}{\operatorname{Var}}
\newcommand{\im}{\operatorname{Im}}
\newcommand{\re}{\operatorname{Re}}
\newcommand{\dd}{\,\mathrm d}
\newcommand{\norm}[1]{\left\lVert#1\right\rVert}

\newcommand{\inner}[2]{\left\langle#1,#2\right\rangle}
\theoremstyle{plain}
\newtheorem{theorem}{Theorem}[section]
\newtheorem{proposition}[theorem]{Proposition}
\newtheorem{lemma}[theorem]{Lemma}

\theoremstyle{definition}
\newtheorem{assumption}[theorem]{Assumption}
\theoremstyle{remark}
\newtheorem{remark}[theorem]{Remark}
\hypersetup{pdftitle={Finite-time divergence of kinetic-energy fluctuations in a Schrodinger-induction flow},pdfauthor={Weishuo Liu}}
\title{Finite-time divergence of kinetic-energy fluctuations\newline in a Schr\"odinger--induction flow}
\author{Weishuo Liu\\[3pt]
\small School of Mechanics and Engineering Science\\
\small Peking University, Beijing 100871, China\\
\small \href{mailto:liuweishuo@pku.edu.cn}{\texttt{liuweishuo@pku.edu.cn}}}
\date{}
\begin{document}
\maketitle

\begin{abstract}
In this study, we use a constrained Schr\"odinger--induction system to investigate kinetic-energy fluctuations induced by a fluid singularity. Assuming the cited Navier--Stokes construction with zero initial velocity and its exact exterior profile, we construct a periodic trajectory from a constant normalized wave function and zero connection. The prescribed body force admits a smooth space--time extension across the normalized singular time $t=1$. In this construction, the probability density remains constant and the mean matter kinetic energy stays bounded, whereas the gauge-invariant kinetic-energy variance diverges. We prove a lower growth bound of order $(1-t)^{-1/2-5h}$, with $h$ fixed by the source construction, by integrating the fourth power of the complete velocity's magnitude over an exact exterior region where all oscillatory corrections vanish. After fixing a gauge with constant wave function, we identify a unique weak $L^2$ connection limit outside $L^4$ and construct its positive self-adjoint kinetic operator through the associated magnetic form. We further show that the normalized endpoint state belongs to the form domain but not the operator domain: its first kinetic spectral moment is finite, whereas its second is infinite. These results connect fluid concentration to a loss of kinetic-domain regularity and provide a dynamical benchmark for assessing the limits of mean-energy control in hydrodynamic wave formulations.

\end{abstract}

\section{Introduction}
Hydrodynamic wave formulations allow a fluid trajectory to be studied through quantities associated with a wave function. Madelung's density--phase transformation provides the basic connection between Schr\"odinger evolution and hydrodynamic equations \cite{Madelung1927}. Recovering a wave description from hydrodynamic variables also requires the appropriate phase conditions, as emphasized by Wallstrom \cite{Wallstrom1994}. For a scalar wave function without zeros, the phase velocity is locally a gradient. Describing smooth rotational motion therefore requires additional structure, and the evolution of that structure must be specified before a fluid singularity can be interpreted in wave variables.

Several formulations address rotational motion in different ways. Dietrich and Vautherin \cite{Dietrich1985} introduced a scalar Schr\"odinger equation coupled to a rotational field, a structure recently revisited by Cappelli et al.\ \cite{Cappelli2026} for quantum algorithms. Chern et al.\ \cite{Chern2016} instead used a two-component wave function in incompressible Schr\"odinger flow, whose energy contains both fluid kinetic and spin-gradient contributions. For viscous flow, Meng and Yang \cite{MengYang2024NS} developed a nonlinear Schr\"odinger--Pauli representation with imaginary diffusion and additional spin-dependent terms. These approaches give different evolution laws and energy contributions, making the choice of observable part of the singularity question.

In the present study, we adopt the scalar wave function with a dynamical rotational connection. The Dietrich--Vautherin structure is adapted to incompressible flow by replacing the barotropic enthalpy closure with the periodic Leray pressure and retaining the prescribed body force. As a result, the wave state and connection evolve together according to the current fields. This distinguishes the dynamical formulation from reconstruction of an external potential using a prescribed future current \cite{Farzanehpour2016}; the kinematic reconstruction itself is established.

The use of wave or spin variables to study fluid regularity has several precedents. Merle et al.\ \cite{Merle2022} used compressible Euler implosion in nonlinear Schr\"odinger blowup constructions. Ohkitani \cite{Ohkitani2017} introduced imaginary-time variables for Navier--Stokes regularity diagnostics, whereas Meng and Yang \cite{MengYang2024SpinEuler} derived an ideal-flow regularity criterion involving the spin Laplacian. The compressible Schr\"odinger-type formulation of Beattie et al.\ \cite{Beattie2026} provides another related setting. These results motivate asking which observable records the loss of regularity in a specified wave representation. Here we examine the gauge-invariant fluctuations of the covariant kinetic-energy operator in a real-time, constant-density model.

The first and second kinetic spectral moments distinguish two levels of control. For a positive observable, a finite first moment requires membership in its form domain, whereas a finite second moment requires membership in its operator domain \cite[Sec.~3.1]{Teschl2014}. In the representation used here, the mean kinetic energy depends on the second spatial moment of the velocity, while its variance also depends on the fourth. Consequently, bounded fluid kinetic energy and unbounded velocity alone do not determine the kinetic variance. For the trajectory studied below, an exact exterior profile supplies the additional information needed to prove its divergence. At the endpoint, the magnetic-form approach for rough connections \cite{Simon1979} gives a precise operator interpretation of the result. Existing domain-preservation results for fixed fields \cite{Boil2020} and stability results for time-dependent Hamiltonians under regularity assumptions \cite{Balmaseda2024} do not directly determine this dynamically generated rough endpoint.

Assuming the Navier--Stokes construction stated in Ref.~\cite{OpenAI2026} and its specified local properties, we construct a trajectory from constant normalized wave data and zero connection. The external source supplies the fluid existence result; its complete proof is not independently verified here. For this trajectory, we prove an explicit lower growth bound for the kinetic variance, although the probability density remains constant and the mean matter kinetic energy stays bounded. We then identify the limiting connection and prove that its kinetic operator has a normalized state with finite first and infinite second spectral moments. Together, the rate and endpoint result connect fluid concentration with kinetic regularity beyond the mean-energy level within the established hydrodynamic representation.

Section 2 presents the constrained wave--induction model. Sections 3 and 4 establish the moment estimate and endpoint-domain result. Section 5 discusses their physical implications, including the role of the full generator and of high-energy dispersion, and Section 6 reports the smooth-field consistency checks.

\section{A closed wave--induction representation}
Let $M$ be a flat three-torus of volume $V$, set $n_0=V^{-1}$, and fix $\eta=\hbar/m>0$ and $\nu>0$. Energies are written per unit mass: the physical matter kinetic operator is $m\K_A$. For a real connection $A$ define
\begin{equation}
 D_A=-i\eta\nabla+A,\qquad \K_A=\frac12D_A^2,
 \qquad j=\re(\overline\psi D_A\psi),\qquad u=j/n_0.
 \label{eq:definitions}
\end{equation}
The connection has velocity units, and $p$ below is kinematic pressure. We consider smooth fields on the constraint manifold
\begin{equation}
 |\psi|^2=n_0,\qquad \diver u=0.
 \label{eq:constraints}
\end{equation}
Given the real body force $f(x,t)$, the pressure is determined from the current fields by
\begin{equation}
 -\Delta p[u]=\diver\big((u\cdot\nabla)u-f\big),\qquad \int_Mp[u]\dd x=0.
 \label{eq:pressure}
\end{equation}
The right side has zero mean on the torus, which makes the Poisson problem solvable; the integral condition fixes its additive constant. The wave and connection then satisfy
\begin{align}
 i\eta\partial_t\psi&=(\K_A+\Phi)\psi,
 \label{eq:wave}\\
 \partial_t A&=u\times(\curl A)-\nu\curl\curl A+f+\nabla(\Phi-p[u]).
 \label{eq:induction}
\end{align}
The real scalar $\Phi$ specifies the gauge. Choosing $\Phi=p[u]$ closes the evolution on the constraint manifold \eqref{eq:constraints}, since the pressure is determined by the current fields and the body force is prescribed. The rotational equation retains the hydrodynamic induction structure \cite{Cappelli2026}. Thus the model gives an evolution law for both the wave state and the connection. Its response to arbitrary density perturbations is outside the present constrained analysis.

For real smooth periodic $\lambda(x,t)$, the transformations
\begin{equation}
 \psi'=e^{i\lambda/\eta}\psi,\qquad A'=A-\nabla\lambda,
 \qquad \Phi'=\Phi-\partial_t\lambda
 \label{eq:gauge}
\end{equation}
leave $u$, $j$ and $\curl A$ invariant. In fact
$D_{A'}(e^{i\lambda/\eta}g)=e^{i\lambda/\eta}D_Ag$.
The gradient term in \eqref{eq:induction} makes the connection equation covariant under time-dependent transformations, since both sides change by $-\nabla\partial_t\lambda$. Omitting this term restricts the equation to its fixed-gauge form.

\begin{proposition}[Exact transfer on the smooth interval]
\label{prop:transfer}
Let $(u,p)$ be a smooth real periodic solution on $0\le t<T$ of
\begin{equation}
 \partial_tu+(u\cdot\nabla)u=-\nabla p+\nu\Delta u+f,
 \qquad\diver u=0,\qquad\int_Mp\dd x=0.
 \label{eq:NS}
\end{equation}
For any real smooth periodic $S_0$, define
\begin{equation}
 S(x,t)=S_0(x)-\int_0^t\left(\frac12|u(x,s)|^2+p(x,s)\right)\dd s,
 \quad \psi=\sqrt{n_0}e^{iS/\eta},\quad A=u-\nabla S.
 \label{eq:lift}
\end{equation}
These fields solve \eqref{eq:constraints}--\eqref{eq:induction} with $\Phi=p$ throughout $[0,T)$. If $u(0)=0$ and $S_0=0$, their initial values are $\psi(0)=\sqrt{n_0}$ and $A(0)=0$. Conversely, every smooth constrained solution of \eqref{eq:wave}--\eqref{eq:induction} induces \eqref{eq:NS}.
\end{proposition}
\begin{proof}
For each $T_*<T$, the integrand in \eqref{eq:lift} is smooth on a compact space--time set. Thus $S$, $\psi$, and $A$ are smooth and periodic on that interval. These intervals exhaust $[0,T)$; no limiting phase at $T$ is assumed. Direct differentiation gives
\begin{equation}
 D_A\psi=u\psi,\qquad D_A^2\psi=(|u|^2-i\eta\diver u)\psi=|u|^2\psi.
 \label{eq:action}
\end{equation}
Therefore $i\eta\psi_t=(|u|^2/2+p)\psi=(\K_A+p)\psi$. Also $\curl A=\curl u=:\omega$, and
\[
 A_t=u_t+\nabla(|u|^2/2+p)=u\times\omega+\nu\Delta u+f
 =u\times\omega-\nu\curl\omega+f.
\]
This is \eqref{eq:induction} in the fixed gauge. Equation \eqref{eq:pressure} follows by taking the divergence of \eqref{eq:NS}.

For the converse, the nonzero constant modulus of $\psi$ permits a local real phase. Its phase one-form $b=\eta\im(\overline\psi\nabla\psi)/n_0$ is globally defined and curl-free, irrespective of a possible spatial phase winding. Locally \eqref{eq:wave} and \eqref{eq:action} give $b_t=-\nabla(|u|^2/2+\Phi)$. Combining this with \eqref{eq:induction} and $u=b+A$ gives \eqref{eq:NS}. All expressions agree on overlapping phase charts.
\end{proof}

Taking $\lambda=-S$ in \eqref{eq:gauge} gives the useful constant-phase gauge
\begin{equation}
 \psi=\psi_0:=V^{-1/2},\qquad A=u,\qquad\Phi=-|u|^2/2.
 \label{eq:constantgauge}
\end{equation}
With this gauge choice, \eqref{eq:induction} reduces to \eqref{eq:NS}. The trajectory therefore inherits its growth mechanism from the fluid evolution, while the constant wave function allows the kinetic moments to be expressed directly in terms of velocity. This relation is used below to distinguish mean-energy control from control of the fluctuations.

\section{Kinetic moments and finite-time growth}
For each $t<T$, the smooth real connection $A$ defines a nonnegative self-adjoint magnetic Laplacian $\K_A$ on $H^2(M)$. The pressure is real and bounded at that time, so $\K_A+p$ is self-adjoint on the same domain. These properties justify the moment calculations on the smooth interval; the endpoint requires the separate construction in Section 4.

\begin{proposition}[First and second kinetic spectral moments]
\label{prop:moments}
For every smooth constrained state,
\begin{align}
 \mu_1(t):=\inner\psi{\K_A\psi}&=\frac{n_0}{2}\norm u_2^2,
 \label{eq:mean}\\
 \Var_\psi(\K_A)&=\frac{n_0}{4}\norm u_4^4-
                     \frac{n_0^2}{4}\norm u_2^4.
 \label{eq:variance}
\end{align}
Both statements are invariant under \eqref{eq:gauge}.
\end{proposition}
\begin{proof}
Equation \eqref{eq:action} gives $\K_A\psi=|u|^2\psi/2$. The first moment is its inner product with $\psi$. The second spectral moment is
$\norm{\K_A\psi}_2^2=n_0\norm u_4^4/4$, by self-adjointness and the spectral theorem. Subtracting $\mu_1^2$ proves \eqref{eq:variance}. Gauge covariance is unitary conjugation of the kinetic operator together with transformation of its state.
\end{proof}

Equation \eqref{eq:variance} shows that the fourth velocity moment is the additional quantity needed to control the kinetic variance. The operator still contains derivatives, even though its action on this particular state has the simple form \eqref{eq:action}. More specifically, setting $g=|u|^2/2$ gives the third moment
\begin{equation}
 \inner\psi{\K_A^3\psi}
 =\frac12\norm{D_A(g\psi)}_2^2
 =n_0\int_M\left(g^3+\frac{\eta^2}{2}|\nabla g|^2\right)\dd x.
 \label{eq:third}
\end{equation}
The gradient contribution shows why the first two moment identities do not identify the full kinetic spectral distribution with the spatial distribution of $|u|^2/2$. The lower-bound argument below uses the second spectral moment, for which the velocity formula is exact.

For the singularity input and subsequent estimates, we use dimensionless source variables. Let $t_{\rm ref}$ be the physical singular time and choose $L_{\rm ref}^2=\nu_{\rm phys}t_{\rm ref}$. With $x_{\rm phys}=L_{\rm ref}x$ and $t_{\rm phys}=t_{\rm ref}t$, the dimensionless viscosity and singular time are $\nu=T=1$, and $\eta$ henceforth denotes $\eta_{\rm phys}t_{\rm ref}/L_{\rm ref}^2=\hbar/(m\nu_{\rm phys})$. Velocities, connections, and the comparison parameter $c$ are scaled by $L_{\rm ref}/t_{\rm ref}$, kinematic pressures and specific energies by $L_{\rm ref}^2/t_{\rm ref}^2$, and the force by $L_{\rm ref}/t_{\rm ref}^2$. Normalization is preserved by $V_{\rm phys}=L_{\rm ref}^3V$ and $\psi_{\rm phys}=L_{\rm ref}^{-3/2}\psi$. All similarity coordinates below are dimensionless; restoring fixed physical scales changes the constants, but not the growth exponents or domain conclusions.

\begin{assumption}[Specific Navier--Stokes input]
\label{ass:source}
At viscosity $\nu=1$, assume a smooth solution of \eqref{eq:NS} on a sufficiently large flat three-torus for $0\le t<1$, with $u(0)=0$, force smooth through $t=1$, and
\begin{equation}
 \sup_{t<1}\norm{u(t)}_2\le C_E<\infty.
 \label{eq:energyinput}
\end{equation}
There is a single point $x_*=0$ away from which $u$ has a locally smooth one-sided limit. Let $h\in(0,1/100)$ be the fixed parameter supplied by the source construction, and put $d=1/2-h$. In local cylindrical coordinates $(r,\vartheta,z)$ near this point, put $\tau=1-t$ and define the positive $q$ by
\begin{equation}
 q-z^2q^{2h}=\tau,\qquad X=\frac{r^2}{2q}.
 \label{eq:similarity}
\end{equation}
For $q<q_*$ and $X\ge X_{\rm ext}$, sufficiently near $(x_*,1)$, the complete velocity is exactly
\begin{equation}
 u(r,\vartheta,z,t)=r^{-1-2h}F(\tau/r^2)e_\vartheta,
 \label{eq:exterior}
\end{equation}
where $F$ is continuous and strictly positive on $[0,\infty)$, and $F(0)>0$.
\end{assumption}

The source \cite{OpenAI2026} asserts the existence of these fields for the fixed $h$ in Assumption~\ref{ass:source}. The required properties are given in its Theorem 1.1, Theorem 3.1(iii), Proposition 9.9 Step 4, Lemmas 10.2--10.3, and Lemma A.6. A torus with a fundamental cell containing the complete fixed spatial support, with a positive boundary margin, allows periodization without changing the local coordinates. In the source notation $F(s)=2^{1/2+h}c_\infty H(4s)$, where $c_\infty>0$ and $H(0)=1$.

The localization equals one near the singular point. All oscillatory and mean corrections vanish in the exterior under consideration, so \eqref{eq:exterior} gives the complete velocity there. This exact formula provides the information needed for the fourth-moment estimate; a velocity-supremum bound alone would be insufficient. The construction vector potential used by the source is distinct from the connection $A$ of the wave model. Appendix~\ref{app:source} lists the individual source dependencies.

\begin{lemma}[Fourth-moment growth from the exact exterior]
\label{lem:shell}
Under Assumption~\ref{ass:source}, there are $c>0$ and $\tau_0>0$ such that
\begin{equation}
 \norm{u(1-\tau)}_4^4\ge c\tau^{-1/2-5h}\qquad(0<\tau<\tau_0).
 \label{eq:L4lower}
\end{equation}
\end{lemma}
\begin{proof}
Fix $R$ with $R^2>8X_{\rm ext}$ and use the cylindrical shell
\[
 S_\tau=\{R\sqrt\tau\le r\le2R\sqrt\tau,\ |z|\le\tau^d,\ 0\le\vartheta<2\pi\}.
\]
Writing $z=\tau^d\zeta$ and $q=\tau Q$ in \eqref{eq:similarity} gives
$Q-\zeta^2Q^{2h}=1$, with $|\zeta|\le1$. Its left side is increasing for $Q\ge1$, since its derivative is at least $1-2h>0$ there. Its value at $Q=4$ exceeds one. Hence $1\le Q\le4$ and $X\ge R^2/8>X_{\rm ext}$. Taking $\tau$ sufficiently small also ensures $q<q_*$ and puts the entire shell where the localization is one.

On this shell, $\tau/r^2\in[1/(4R^2),1/R^2]$. The positive minimum of $F$ on this fixed interval and \eqref{eq:exterior} give $|u|\ge c_R\tau^{-1/2-h}$. The shell volume is exactly
\[
 |S_\tau|=(2\pi)(2\tau^d)\int_{R\sqrt\tau}^{2R\sqrt\tau}r\dd r
 =6\pi R^2\tau^{1+d}.
\]
Integrating the fourth power of the velocity lower bound gives
$\norm u_4^4\ge c\tau^{1+d-4(1/2+h)}=c\tau^{-1/2-5h}$.
\end{proof}

\begin{theorem}[Dynamically generated kinetic fluctuations]
\label{thm:fluctuations}
Assumption~\ref{ass:source} yields a solution of the closed fixed-gauge wave--induction model with initial data $\psi_0=V^{-1/2}$ and $A_0=0$, a body force smooth through $t=1$, and constant normalized probability density. Its mean matter kinetic energy per unit mass is uniformly bounded, while
\begin{equation}
 \Var_{\psi(t)}(\K_{A(t)})\ge c_1(1-t)^{-1/2-5h}-c_2\longrightarrow\infty.
 \label{eq:varlower}
\end{equation}
Here $h$ is the fixed source parameter in Assumption~\ref{ass:source}. The exponent is a proved lower-bound exponent; no matching full-field upper bound is asserted.
\end{theorem}
\begin{proof}
Proposition~\ref{prop:transfer} gives the closed trajectory and initial data. Equations \eqref{eq:mean} and \eqref{eq:energyinput} bound the first moment by $n_0C_E^2/2$. Proposition~\ref{prop:moments} and Lemma~\ref{lem:shell} give \eqref{eq:varlower}, with $c_2=n_0^2C_E^4/4$.
\end{proof}

The estimate also gives a lower bound on the kinetic standard deviation: for sufficiently late times, it is at least a positive multiple of $(1-t)^{-1/4-5h/2}$. Thus the fluctuations grow without bound while the mean matter kinetic energy remains bounded. The latter bound concerns the kinetic observable of the wave state and includes no additional energy functional for the dynamical connection.

\section{Endpoint kinetic operator and domain loss}
The variance estimate establishes the growth of fluctuations before the singular time. We next identify the limiting connection so that the endpoint can be described by a fixed kinetic operator. Its form and operator domains then determine, respectively, whether the first and second kinetic spectral moments remain finite.

\begin{lemma}[The endpoint connection]
\label{lem:endpointfield}
Under Assumption~\ref{ass:source}, the complete family has a unique weak endpoint
\begin{equation}
 u(t)\rightharpoonup u_T\quad\text{in }L^2(M),\qquad
 u_T\in L^2(M;\R^3)\setminus L^4(M;\R^3),\qquad\diver u_T=0.
 \label{eq:endpointfield}
\end{equation}
\end{lemma}
\begin{proof}
Define $u_T$ by the locally smooth limit off $x_*$ and arbitrarily at that point. Fatou's lemma and \eqref{eq:energyinput} give $\norm{u_T}_2\le C_E$. For a smooth test vector $\xi$, split the pairing of $u(t)-u_T$ with $\xi$ into a ball $B_\epsilon(x_*)$ and its complement. On the ball its absolute value is at most $2C_E\norm{\xi1_{B_\epsilon}}_2$, which tends to zero with $\epsilon$. On the complement use the local smooth convergence. Density of smooth tests in $L^2$ and the uniform norm bound give weak convergence of the full family. Testing with gradients gives the distributional divergence assertion.

For fixed $z\ne0$, the positive solution branch of \eqref{eq:similarity} satisfies $q(z,\tau)\to|z|^{1/d}$ as $\tau\downarrow0$. Choose $b>0$ with $1/(2b^{1/d})>X_{\rm ext}$, and $r_0$ small enough that
\[
 W=\{0<r<r_0,\ |z|<br^{2d},\ 0\le\vartheta<2\pi\}
\]
lies inside the relevant localization. Every fixed interior point with $z\ne0$ eventually belongs to the exact exterior as $\tau\downarrow0$. Thus
$u_T=F(0)r^{-1-2h}e_\vartheta$ almost everywhere on $W$. In particular,
\begin{equation}
 \int_W|u_T|^p\dd x
 =4\pi bF(0)^p\int_0^{r_0}r^{2-2h-p(1+2h)}\dd r.
 \label{eq:wedge}
\end{equation}
For $p=4$ the exponent is $-2-10h<-1$, giving divergence. For $p=2$ the local exponent is $-6h>-1$, consistent with the already proved global $L^2$ bound. No claim about the full endpoint's other integrability thresholds is needed.
\end{proof}

For a rough connection, the magnetic quadratic form provides the natural starting point \cite{Simon1979}. The closed-form representation theorem then supplies the associated self-adjoint operator \cite[Theorem~2.14]{Teschl2014}. We give the periodic constant-state criterion explicitly, including the regularity assumptions needed for both directions.

\begin{proposition}[The constant-state domain criterion]
\label{prop:domain}
Let $A\in L^2(M;\R^3)$ be distributionally divergence-free. On $C^\infty(M;\C)$ define the nonnegative magnetic form
\begin{equation}
 q_A^0(f,g)=\frac12\inner{-i\eta\nabla f+Af}{-i\eta\nabla g+Ag}.
 \label{eq:roughform}
\end{equation}
It is closable; its closure $q_A$ defines a unique nonnegative self-adjoint operator $\K_A$. For $\psi_0=V^{-1/2}$,
\begin{equation}
 \psi_0\in\D(\K_A^{1/2}),\qquad
 q_A[\psi_0]=\frac1{2V}\norm A_2^2,
 \qquad
 \psi_0\in\D(\K_A)\ \Longleftrightarrow\ A\in L^4(M).
 \label{eq:domaincriterion}
\end{equation}
\end{proposition}
\begin{proof}
Let $D_A^0=-i\eta\nabla+A$ initially map smooth scalar functions into vector $L^2$. This is well-defined since smooth functions are bounded. Suppose $f_j\to0$ in scalar $L^2$ and $D_A^0f_j\to F$ in vector $L^2$. For every smooth vector test $\xi$, integration by parts gives
\[
 \inner\xi{D_A^0f_j}=\inner{-i\eta\diver\xi+A\cdot\xi}{f_j}\longrightarrow0.
\]
The function on the left slot of the right pairing is in $L^2$. Density of the tests implies $F=0$, proving closability of $D_A^0$ and hence of \eqref{eq:roughform}. Its closure has form domain $Q_A=\D(\overline D_A^0)$ and value $q_A[f]=\norm{\overline D_A^0f}_2^2/2$. The closed-form representation theorem gives $\K_A$ with $Q_A=\D(\K_A^{1/2})$.

The constant $\psi_0$ is in the original form core, with $D_A^0\psi_0=A\psi_0$, proving the first two claims. For a smooth scalar test $\varphi$, distributional divergence-freeness gives
\begin{equation}
 q_A(\varphi,\psi_0)
 =\frac{\psi_0}{2}\int_M\big(i\eta\nabla\overline\varphi\cdot A+
                       \overline\varphi|A|^2\big)\dd x
 =\frac{\psi_0}{2}\int_M\overline\varphi|A|^2\dd x.
 \label{eq:domaintest}
\end{equation}
If $\psi_0\in\D(\K_A)$, this distribution must be represented by an $L^2$ function, namely $\K_A\psi_0$. Therefore $|A|^2\in L^2$, which is $A\in L^4$. Conversely, if $A\in L^4$, set $g=\psi_0|A|^2/2\in L^2$. Equation \eqref{eq:domaintest} extends from the form core to all $\varphi\in Q_A$ by approximation in the form norm. The definition of the operator associated with $q_A$ then gives $\psi_0\in\D(\K_A)$ and $\K_A\psi_0=g$.
\end{proof}

The quadratic form is needed at this regularity because the second-order differential expression for an $L^2$ connection need not map every smooth function into $L^2$. The construction therefore begins with the first-order operator and its closed form. Under the stated assumptions, no identification of its form domain with ordinary $H^1$ is made.

\begin{theorem}[Endpoint kinetic-domain loss]
\label{thm:endpoint}
For the trajectory in Theorem~\ref{thm:fluctuations}, use the constant-phase gauge \eqref{eq:constantgauge} and set $A_T=u_T$. The closed form in \eqref{eq:roughform} defines a positive endpoint kinetic operator $\K_T$, for which
\begin{equation}
 \boxed{\psi_0\in\D(\K_T^{1/2})\setminus\D(\K_T).}
 \label{eq:endpointdomain}
\end{equation}
The endpoint state is normalized and has finite first kinetic spectral moment and infinite second spectral moment and variance. These conclusions remain true under a smooth periodic endpoint gauge transformation applied to both state and operator.
\end{theorem}
\begin{proof}
Combine Lemma~\ref{lem:endpointfield} with Proposition~\ref{prop:domain}. If $\mu$ is the spectral probability measure of $\K_T$ in $\psi_0$, the spectral theorem gives
\[
 \int_0^\infty s\dd\mu(s)=q_{A_T}[\psi_0]=\frac1{2V}\norm{u_T}_2^2<\infty,
 \qquad \int_0^\infty s^2\dd\mu(s)=\infty.
\]
Subtracting a finite mean does not make the second moment finite, so the central variance is infinite. The first moment here is defined by the form or spectral integral; it is not an $L^2$ pairing with an undefined $\K_T\psi_0$.

For a real smooth periodic $\lambda$, multiplication by $G=e^{i\lambda/\eta}$ maps the smooth form core bijectively to itself, and $q_{A_T-\nabla\lambda}^0(Gf,Gg)=q_{A_T}^0(f,g)$. Taking closures gives unitary equivalence of the associated positive operators and transport of both domains. This proves the stated covariance without assuming a rough endpoint gauge calculus.
\end{proof}

\section{Physical interpretation and implications}
In the present construction, kinetic fluctuations become singular while the probability density remains uniform, $|\psi|^2=1/V$, and the mean matter kinetic energy remains bounded. The distinction persists at the identified endpoint: the normalized state has finite first kinetic spectral moment, whereas its second moment is infinite. In the gauge \eqref{eq:constantgauge}, the Hilbert-space state is constant and the evolving connection determines the changing kinetic observable. The construction therefore shows how higher kinetic moments can record fluid concentration that is not excluded by density or mean-energy control.

This conclusion is specific to the covariant kinetic observable. In the constant-phase gauge, the scalar term is $\Phi=-|u|^2/2$, so $(\K_u+\Phi)\psi_0=0$ for every $t<1$. The kinetic variance can therefore diverge even though the full generator has zero variance in this state. More generally, a time-dependent gauge transformation adds a time-derivative term to the conjugated Hamiltonian. These facts prevent interpreting \eqref{eq:varlower} as a full-Hamiltonian variance blowup or as a quantum-speed-limit result.

For a material with a fixed nonzero microscopic length, continued contraction eventually removes the scale separation required by continuum hydrodynamics. Writing the variables on a continuous configuration space does not extend the validity of the evolution law to smaller scales. In the constant-phase gauge, $A=u$, so the connection retains the singular hydrodynamic structure. The wave--induction system remains a constrained nonlinear representation of viscous flow; it is neither an isolated linear quantum system nor a self-consistent Maxwell--Schr\"odinger model. This interpretation agrees with the distinction made by Dietrich and Vautherin \cite{Dietrich1985} between a formal wave representation and a quantum derivation. Continuity of the wave function consequently gives no basis for claiming greater physical realizability of the singularity.

\begin{remark}[Dependence on the high-energy dispersion]
\label{rem:dispersion}
The variance divergence concerns the nonrelativistic quadratic kinetic observable. For fixed $c>0$, define the operator-square-root comparison observable
\begin{equation}
 R_c=c^2\left(\sqrt{I+2\K_A/c^2}-I\right).
 \label{eq:sqrtkinetic}
\end{equation}
This is one standard magnetic square-root prescription \cite{Hiroshima2017}. Its spectral function $r_c(s)$ satisfies
$r_c(s)^2+2c^2r_c(s)=2c^2s$ for $s\ge0$. Hence every normalized $\psi\in\D(\K_A^{1/2})$ belongs to $\D(R_c)$ and
\begin{equation}
 \Var_\psi(R_c)\le\norm{R_c\psi}_2^2\le2c^2q_A[\psi].
 \label{eq:sqrtvariance}
\end{equation}
In fact $\D(R_c)=\D(\K_A^{1/2})$: the reverse inclusion follows by integrating $s=r_c(s)+r_c(s)^2/(2c^2)$ and using Cauchy--Schwarz. Thus the bounded first kinetic moment in our trajectory bounds this comparison variance, including at the endpoint, even though the variance of $\K_A$ diverges. The singularity is consequently sensitive to the assumed high-energy dispersion. This is an observable comparison on the same state and connection; it does not construct a relativistic coupled evolution or show that this trajectory solves a Dirac, Maxwell--Schr\"odinger, or square-root evolution equation. A bounded mean nonrelativistic kinetic energy alone is not a uniform low-energy justification for its higher moments.
\end{remark}
Despite its infinite second kinetic moment, the endpoint state remains normalized and has finite kinetic form energy. For the fixed positive operator $\K_T$, the auxiliary group $e^{-is\K_T}\psi_0$ defines a continuous $L^2$ orbit for all $s$, whereas exclusion from $\D(\K_T)$ removes the strong derivative at $s=0$ \cite[Theorem~5.1]{Teschl2014}. Thus the domain result identifies a loss of strong differentiability for this auxiliary evolution. It does not supply a continuation of the original coupled system.

The endpoint operator has been constructed directly from the weak $L^2$ limit of the connection. Such coefficient convergence alone does not imply strong-resolvent or Mosco convergence, and neither type of operator convergence is established here. No uniqueness or stability result for a singular endpoint evolution, or continuation of the coupled system through $t=1$, is established either. These questions require information beyond the limiting coefficient and the constant-state domain criterion.

The prescribed force admits a smooth space--time extension across the event, whereas the evolving connection becomes singular. Hence the singular behavior in this construction is generated during the evolution from regular initial fields; it is not introduced by a singular initial connection or a singular body force. The bounded mean controls matter kinetic energy and supplies no bound for an additional field-energy functional. Any prediction for a microscopic system would require a specified model, dispersion relation and range of valid scales.

Within these conditions, the contribution is the dynamical connection between fluid concentration, the growth of kinetic fluctuations and an identified endpoint domain loss. The rotational representation, the fluid existence mechanism and the static form/operator-domain distinction are inherited structures. We prove how their combination produces the endpoint from constant wave data and zero connection, with a lower rate obtained from the complete velocity in an exact exterior region. This gives a benchmark for assessing what mean-energy bounds can and cannot control in hydrodynamic wave models, and for formulating further questions about higher moments and operator convergence.

\section{Consistency checks and verification}
We checked the wave--induction equations and kinetic-moment identities using a smooth periodic flow with known analytical derivatives. The same quantities were evaluated numerically by Fourier differentiation, including after a time-dependent gauge transformation. Details of the test fields, numerical procedure and results are provided in the Supplemental Material \cite{Supplemental}.

For all nine test cases, the numerical results agree with the analytical identities to within $10^{-13}$. This agreement supports the implementation of the covariant operators and gauge transformations used above. The tests concern prescribed smooth fields; the singularity rate and endpoint-domain results follow from the analytical proofs in Sections 3 and 4.

\paragraph{Research assistance and verification.}
OpenAI ChatGPT and Codex assisted literature discovery, exploratory derivations, scientific interpretation, drafting, code generation and consistency checks. The author directed the research question and scope. Verification was assisted by AI. The human author is responsible for the final manuscript.

\section*{Data and code availability}
The numerical methods, code and data accompany this work as arXiv ancillary files and as Supplemental Material for journal review \cite{Supplemental}. The sources of the external Navier--Stokes construction are identified in Ref.~\cite{OpenAI2026} and Appendix~\ref{app:source}.

\section*{Acknowledgments and declarations}
The author reports no specific funding or dedicated support and no relevant financial or nonfinancial competing interests. The research and manuscript benefited from the AI assistance described with the verification methods above.

\appendix
\section{External singularity input and source audit}
\label{app:source}
Assumption~\ref{ass:source} is the sole external singularity input. The following properties are taken from the analysis manuscript \cite{OpenAI2026}; the subsequent fourth-moment and endpoint-domain calculations are proved here.
\begin{enumerate}
\item Theorem 1.1 supplies the viscosity-one smooth solution before $t=1$, zero initial velocity, a fixed compact spatial support, bounded $L^2$ kinetic energy, and a force smooth through the singular time. Proposition 10.1 supplies the localization used near the singular point, and Lemma 10.3 extends the force in the same spatial support. A large periodic cell containing the common support strictly in its interior permits disjoint periodization without rescaling the local coordinates.
\item Theorem 3.1(ii), Proposition 9.9 Steps 3--4, and Lemma 10.2 give one-sided local smooth limits away from the singular point, after the localization. Lemma 10.2 explicitly combines the positive-$q$ region with the positive-radius exterior. These are the inputs to the full-family weak endpoint argument of Lemma~\ref{lem:endpointfield}.
\item Theorem 3.1(iii), Eq.~(3.5), and Proposition 9.9 Step 4 give the exact exterior velocity \eqref{eq:exterior}. The annular wave and mean corrections vanish there. This is why Lemma~\ref{lem:shell} applies to the complete velocity without a core-dominance assumption.
\item Lemma A.6 and Eq.~(A.32) give positivity of the heat factor and $H(0)=1$. Together with the positive exterior coefficient, these give the compact-interval lower bound for $F$ and the nonzero endpoint coefficient used in \eqref{eq:wedge}.
\end{enumerate}
The external input is taken from the 166-page manuscript dated 8 September 2026 and consulted on 9 September 2026. The accompanying source-dependency record identifies the specific statements used here \cite{Supplemental}. The source authors also provide a Lean formalization. We have not run that project or independently verified the complete external existence proof; the transfer, fourth-moment and endpoint-domain arguments are proved in the present article.

\end{document}